\documentclass[11pt,hidelinks,a4paper]{article}

\usepackage{algpseudocode}
\usepackage{algorithm}
\usepackage{amsfonts}
\usepackage{amsmath}
\usepackage{amssymb}
\usepackage{amstext}
\usepackage{amsthm}
\usepackage[toc,page]{appendix}
\usepackage{authblk}
\usepackage{booktabs}
\usepackage{caption}
\usepackage{cite}
\usepackage{datetime}
\usepackage{enumitem}
\usepackage{framed}
\usepackage{fullpage}
\usepackage{graphics}
\usepackage{graphicx}
\usepackage[latin1]{inputenc}
\usepackage{mathrsfs}
\usepackage{multirow}
\usepackage{pdflscape}
\usepackage{pgfplots}
\usepackage{pifont}
\usepackage{rotating}
\usepackage{setspace}
\usepackage{subcaption}
\usepackage{tikz}
\usepackage{times}
\usepackage{units}
\usepackage{url}
\usepackage{xcolor}
\usepackage{hyperref}

\usepackage[textsize=small,color=blue!20]{todonotes}

\usetikzlibrary{shapes,arrows,plotmarks,matrix,positioning,fit,calc,3d}
\pgfplotsset{compat=1.8}
\theoremstyle{definition}

\newcommand{\iid}{\stackrel{\mathrm{iid}}{\sim}}
\newcommand{\as}{\stackrel{\mathrm{\text{a.s.}}}{=}}

\newtheorem{theorem}{Theorem}[section]
\newtheorem{lemma}[theorem]{Lemma}
\newtheorem{proposition}[theorem]{Proposition}
\newtheorem{corollary}[theorem]{Corollary}
\newtheorem{remark}[theorem]{Remark}
\newtheorem{mydef}[theorem]{Definition}

\newtheorem*{mymodel*}{The General Spline Model}

\newenvironment{assumptions}
{\par\addvspace{\baselineskip} \noindent\textbf{Assumptions:}
}
{\paragraph*{}}
\renewenvironment{proof}{{\bfseries Proof.}}{\qed}

\def\keywords{\vspace{.5em}
{\textbf{Keywords}:\,\relax%
}}

\newcounter{num}
\renewcommand{\thenum}{\arabic{num}}

\newlength\figureheight
\newlength\figurewidth

\begin{document}

\title{Pointwise Convergence in Probability\\of General Smoothing Splines}
\author{Matthew Thorpe}
\affil{Department of Mathematics, Carnegie Mellon University,\\ Pittsburgh, PA 15213, United States}
\author{Adam M. Johansen}
\affil{Department of Statistics, University of Warwick, Coventry, CV4 7AL, United Kingdom }
\date{March 2017}
\maketitle

\begin{abstract}
Establishing the convergence of splines can be cast as a variational problem which is amenable to a $\Gamma$-convergence approach.
We consider the case in which the regularization coefficient scales with the number of observations, $n$, as $\lambda_n=n^{-p}$.
Using standard theorems from the $\Gamma$-convergence literature, we prove that the
general spline model is consistent in that estimators converge in a
sense slightly weaker than weak convergence in probability for $p\leq 
\frac{1}{2}$.
Without further assumptions we show this rate is sharp.
This differs from rates for strong convergence using Hilbert scales where one can often choose $p>\frac{1}{2}$.
\end{abstract}


\section{Introduction \label{sec:Intro}}

Given a Hilbert space, $\mathcal{H}$, with dual $\mathcal{H}^*$, the general
spline problem \cite{kimeldorf71},\cite{wahba90} is to recover $\mu^\dagger\in\mathcal{H}$ from
observations, $\{(L_i,y_i)\}_{i=1}^n\subseteq \mathcal{H}^*\times\mathbb{R}$, and the model
\begin{equation} \label{eq:Intro:Data}
y_i = L_i \mu^\dagger + \epsilon_i,
\end{equation}
where $\epsilon_i$ and $L_i$ are independent random variables taking values in $\mathbb{R}$ and $\mathcal{H}^*$, respectively.
We assume that $\mathcal{H}$ can be decomposed into
$\mathcal{H}=\mathcal{H}_0\oplus\mathcal{H}_1$ where, for $l=0,1$,
$(\mathcal{H}_l,\|\cdot\|_l)$ are themselves both Hilbert spaces.
For example, one may apply the theory to the special spline problem (also referred to as smoothing splines) where $\mathcal{H}=H^m([0,1])$ ($m\geq 1$) is the Sobolev space of degree $m$ and the observation operators are of the form $L_i\mu =\mu(t_i)$ in which $t_i$ is sampled from some distribution over $[0,1]$.
Throughout this paper we refer to~\eqref{eq:Intro:Data} as the general spline model when $L_i\in\mathcal{H}^*$ and $\mathcal{H}$ is any Hilbert space, and the special spline model when $L_i$ is the pointwise evaluation operator and $\mathcal{H}=H^m$.

Establishing convergence and the rate of convergence of estimates $\mu^n$ of $\mu^\dagger$ remains a current area of research~\cite{bissantz04,bissantz07,claeskens09,hall05,kauermann09,lai13,lukas06,wang11}.
These results establish strong convergence, in the sense of convergence with respect to a norm, and related rates of the special spline problem.
Convergence with respect to the norm in the original space is typically not
achievable so convergence results are in weaker topologies (equivalently
larger spaces). This paper fills a gap in the literature by establishing the
convergence of the general spline problem in the original space 
in the sense that  $\forall F \in \mathcal{H}^*$, $F(\mu^n)$ converges in probability to $F(\mu^\dagger)$.
There exist results for pointwise convergence of the special spline problem with equally spaced ($t_i = \frac{i}{n}$) data points~\cite{li08,shen11,xiao12,yoshida12,yoshida14}. 
Our results do not assume data points are equally spaced (we do however require that they are iid) and we consider the general case where $L_i$ are bounded and linear operators (not necessarily pointwise evaluation).

We assume that $\text{dim}(\mathcal{H}_0)=m<\infty$ and $\text{dim}(\mathcal{H}_1)=\infty$.
This can be seen as a multi-scale decomposition of $\mathcal{H}$.
The projection of a function $\mu\in\mathcal{H}$ into the subspace $\mathcal{H}_0$ is a coarse approximation of that function.
Continuing with the special spline example, one can write
\[ \mu(t) = \sum_{i=0}^{m-1} \frac{\nabla^i\mu(0)}{i!} t^i + \int_0^t \frac{(t-u)^{m-1}}{(m-1)!} \nabla^m\mu(u) \; \mathrm{d} u \]
for any $\mu\in H^m$.
The space $\mathcal{H}_0$ is then the space of polynomials of degree at most $m-1$.
Hence $\text{dim}(\mathcal{H}_0)=m$.
Imposing a penalty on the $\mathcal{H}_1$ space, we construct a sequence of estimators $\mu^n$ of $\mu^\dagger$ as the minimizers of
\[ f_n(\mu) = \frac{1}{n}\sum_{i=1}^n |y_i-L_i \mu|^2 + \lambda_n \|\chi_1 \mu\|_1^2 \]
where $\chi_i:\mathcal{H}\to\mathcal{H}_i$ ($i=0,1$) is the projection of $\mathcal{H}$ onto $\mathcal{H}_i$.
This paper addresses the asymptotic behaviour (as $n\to\infty$) of the general spline
problem and in particular how one should choose $\lambda_n$ to ensure $\mu^n$ converges
(in the weak sense that $\forall F \in \mathcal{H}^*$, $F(\mu^n)$ converges in probability to $F(\mu^\dagger)$) to $\mu^\dagger$.
An alternative, but closely related, method is the penalized spline problem, for example~\cite{eilers96}, where the estimate $\mu^\dagger$ is found by minimizing $f_n$ over functions of the form $\mu = \sum_{i=1}^\ell a_i B_i$ where $B_i$ are a set of $B$-splines and penalising the coefficients $a_i$ or derivatives of $\mu$.
Typically $\ell\ll n$ so the complexity of the problem decreases.

There are two bodies of literature on the specification of $\lambda_n$.
On the one hand there are methods which define $\lambda_n$ as the minimizer of some loss function, for example average square error.
This class of techniques includes cross-validation \cite{wahba75}, generalized cross-validation \cite{craven79} and penalized likelihood techniques \cite{hastie90,hurvich98,kou02,mallows73,sakamoto86,wahba85}.
These methods provide a numerical value of $\lambda_n$ for a given $n$ and a given set of data.
In the case of special splines there are many results on the asymptotic behavior of $\lambda_n$ and $\mu^n$ for these methods, see for example \cite{aerts02,cox83,craven79,li87,speckman01,utreras81,utreras83,wand99}.
The alternative approach, and the one we take in this paper, is to choose a
sequence such that the estimates $\mu^n$ converge to $\mu^\dagger$ in an appropriate
sense at the fastest possible rate. This strategy gives a scaling
regime for $\lambda_n$, but it does not in general give specific numerical values of
$\lambda_n$, i.e. it provides the optimal rate of convergence but not the associated
multiplicative constant. 


When considering strong convergence many results in the literature demonstrate $\mu^n\to \mu^\dagger$ in a norm via the use of Hilbert scales --- see, for example,~\cite{cox88,nychka89,ragozin83,speckman85,stone82,utreras85}.
It is not typically possible to obtain strong convergence with respect to the 
original norm and it is common to resort to the use of weaker norms; for example, in the special spline problem, one starts with the space $H^s$ but looks for convergence in $L^2$. 
The alternative, which is pursued in this paper, is to consider modes of
convergence related to weak convergence in the original space, $\mathcal{H}$.

Note that for special splines strong convergence in a larger space is a weaker
result than weak convergence in the original space: by the Sobolev embedding
theorem, weak convergence in $H^s$ implies strong convergence in $L^2$; however, the converse does not hold.

In this paper we show that the estimators of the general spline problem
converge in a sense slightly weaker than convergence weakly in probability in the large data limit, $\mu^n\rightharpoonup \mu^\dagger$, for regularization $\lambda_n$ that scales to zero no faster than $n^{-\frac{1}{2}}$.
In this scaling regime we say that the general spline problem is consistent.
For insufficient regularization the spline estimators may in some sense `blow up'.
In particular for scaling outside this regime we construct (uniformly bounded) observation operators $L_i$ such that $\mathbb{E}\left[\|\mu^n\|^2\right]\to \infty$.
Hence without further assumptions our results are sharp.

We note that these results have practical implications. If we are interested
in estimating $\mu^\dagger$ at a point $t$ then we let $F(\mu)=\mu(t)$ where
$F\in\mathcal{H}^*$. In this setting weak convergence, or the pointwise form
considered in this paper, are the natural modes of convergence to consider.
Whereas, if one is interested in a global approximation of $\mu^\dagger$, then
convergence of $\mu^n - \mu^\dagger$ in an appropriate norm is the more relevant. The two
formulations imply different scaling results for $\lambda_n$.   

There are many results in the ill-posed inverse problems literature that may be applied to the strong convergence of the general spline problem, for brevity we only mention those most relevant to this work.
In \cite{wahba85} two different methods of estimating $\lambda_n$ were compared as $n\to\infty$ using the general spline formulation.
The reproducing kernel Hilbert space setting was used in \cite{kimeldorf70} which also discussed the probabilistic interpretation behind the estimator $\mu^n$.
In \cite{cox88,nychka89} the authors prove the strong convergence and optimal rates for the spline model using an approximation $\frac{1}{n}\sum_{i=1}^n L_i^* L_i \approx U$ where $U$ is compact, positive definite, self-adjoint and with dense inverse.
See also \cite{carroll91,mair96} that consider ill-posed inverse problems without noise using similar methods.
In these papers the scaling regime for $\lambda_n$ is given in terms of the rate of decay of the eigenvalues of the inverse covariance (regularization) operator $\mathcal{C}^{-1}$ (where $\|\cdot\|_1=\|\mathcal{C}^{-1}\cdot\|_{L^2}$).

There are many more recent results addressing the asymptotic properties of
splines, including \cite{claeskens09,hall05,kauermann09,lai13,li08,shen11,wang11,xiao12,yoshida12,yoshida14}. Many
of these recent results concern the asymptotics of penalized splines where one
fixes the number of knot points as apposed to the smoothing spline case where
the number of knots is equal to the number of data points.

It is known that the special spline problem is equivalent to a white noise problem \cite{brown96}.
Strong convergence and rates for the white noise problem have been well studied see, for example, \cite{agapiou13,bissantz07,goldenshluger00} and references therein.

An interesting related result, due to Silverman~\cite{silverman84}, gives the convergence of the smoothing kernel.
That is, we can write the estimator $\mu^n$ of $\mu$ given data $\{(t_i,y_i)\}_{i=1}^n$ in the form
\[ \mu^n(s) = \frac{1}{n} \sum_{i=1}^n K_n(s,t_i) y_i \]
for a Kernel $K_n$ (see Lemma~\ref{lem:Prelim:Splines:Minfn}).
Silverman showed that $K_n(\cdot,t)$ converges to some $K$ uniformly on $[\epsilon,1-\epsilon]$ for every $\epsilon>0$ and each $t$ (the result is valid for the special spline model and penalising the second derivative).
Whilst this result gives intuition into how the kernel behaves it does not imply the convergence of the smoothing spline.
Indeed, the convergence is not valid at the end points $\{0,1\}$ and does not account for randomness in the observations $y_i$.
In other words $K_n(\cdot,t)\to K(\cdot, t)$ does not imply the convergence of $\mu^n$ (or any characterisation of the limit such as we give in this paper as a solution to a variational problem).
Silverman's result is, however, valid for a larger range of $\lambda$ than we have here.
For convergence of the kernel it is enough that $\frac{1}{\lambda} = o(n^{2-\delta})$ for any $\delta>0$.
Our results concerning the pointwise convergence of the smoothing spline hold for $\lambda$ satisfying $\frac{1}{\lambda} = O(n^{\frac12})$.

One advantage of our approach is that we gain intuition in what happens when $\lambda_n\to 0$ too quickly.
Our results show a critical rate, with respect to the scaling of $\lambda_n$, at
which the methodology is ill-posed below this rate and well-posed at or above
this rate. The second advantage of our approach is that, by using the $\Gamma$-convergence framework, as long as we can show that minimizers are uniformly bounded the
convergence follows easily (we also need to show the $\Gamma$-limit is unique, but for our problem this is not difficult).
This is easier than showing, directly, that $\mu^n-\mu^\dagger$ converges to zero.
We are consequently able to employ simpler assumptions than those required by more direct arguments.

The outline of this paper is as follows.
In the next section we introduce some preliminary material.
This starts by defining the notation we use in the remainder of the paper.
We then remind the reader of G\^ateaux derivatives, the $\Gamma$-convergence
framework and the spline methodology respectively. 
Section~\ref{sec:Cons} contains the results for the convergence of the general
spline model under appropriate conditions on the scaling in the regularization
using the $\Gamma$-convergence framework. 
We discuss the special spline model in Section~\ref{sec:Ex}.
\section{Preliminary Material \label{sec:Prelim}}
\subsection{Notation \label{sec:Prelim:Not}}
We use the following standard definitions for rates of convergence.
\begin{mydef}
\label{def:Prelim:Not:Rate}
We define the following.
\begin{enumerate}[label = (\roman*)]
\item For deterministic sequences $a_n$ and $r_n$, where $r_n$ are positive and real valued, we write $a_n=O(r_n)$ if $\frac{a_n}{r_n}$ is bounded.
If $\frac{a_n}{r_n}\to 0$ as $n\to \infty$ we write $a_n=o(r_n)$.
\item For random sequences $a_n$ and $r_n$, where $r_n$ are positive and real valued, we write $a_n=O_p(r_n)$ if $\frac{a_n}{r_n}$ is bounded in probability: for all $\epsilon>0$ there exists $M_\epsilon,N_\epsilon$ such that
\[ \mathbb{P}\left(\left\vert \frac{a_n}{r_n} \right\vert \geq M_\epsilon\right) \leq \epsilon \quad \quad \quad \forall n\geq N_\epsilon. \]
If $\frac{a_n}{r_n}\to 0$ in probability: for all $\epsilon>0$
\[ \mathbb{P}\left(\left\vert \frac{a_n}{r_n} \right\vert\geq \epsilon \right) \to 0 \quad \quad \quad \text{as } n\to \infty \]
we write $a_n=o_p(r_n)$.
\end{enumerate}
\end{mydef}

\begin{mydef}
\label{def:Prelim:Not:Lesssim}
For deterministic positive sequences $a_n$ and $b_n$ we write $a_n\lesssim
b_n$ to mean there exists $M < \infty$ such that $a_n\leq M b_n$ for all $n$.
\end{mydef}

Throughout this paper we say that a sequence of parameter estimators is
consistent if, for any value of the ``parameters'' (splines in our setting),
they converge in the sense made precise in Theorem \ref{thm:Cons:Cons}
to the true value.

We will assume $\epsilon_i$ and $L_i$ are independent sequences of iid random
variables. Our estimators $\mu^n$ are also random variables and therefore we can
reach only probabilistic conclusions about the convergence of $\mu^n$.

We will work on a probability space $(\Omega,\mathcal{F},\mathbb{P})$ rich enough
to support a countably infinite sequence of observations $(L_i,y_i)_{i\geq1}$.
All stochastic quantifiers are taken with respect to $\mathbb{P}$ unless otherwise
stated. It will be convenient to introduce the natural filtration associated
with the marginal sequence $(L_i)$ and we define for $n\in\mathbb{N}$,
$\mathcal{G}_n = \sigma(L_1,\ldots,L_n)$, a sequence of sub-$\sigma$-algebras of
$\mathcal{F}$. We use $\mathbb{E}[\cdot|\mathcal{G}_n]$ to denote a version of the
associated conditional expectation. 

To emphasize the dependence on the realization $\omega\in\Omega$, and hence of the data sequence, of our functionals we write $f_n^{(\omega)}$.

For an operator $U:\mathcal{H}\to\mathcal{H}$ we will use $\text{Ran}(U)$ to denote the range of $U$, i.e.
\[ \text{Ran}(U) = \left\{ \mu\in\mathcal{H} : \exists \nu\in\mathcal{H} \text{ s.t. } U\nu = \mu \right\}. \] 
When $U$ is linear the operator norm is defined by
\[ \|U\|_{\mathcal{L}(\mathcal{H},\mathcal{H})} := \sup_{\|\mu\|\leq 1} \|U\mu\|. \]
We denote the support of a probability measure $\phi$ on a topological space
$\mathcal{I}$ endowed with its Borel $\sigma$-algebra, by $\text{supp}(\phi)$, i.e.
\[ \text{supp}(\phi) = \inf \left\{\mathcal{I}': \mathcal{I}'\subset \mathcal{I}, \mathcal{I}' \text{ is closed, and } \int_{\mathcal{I}\setminus\mathcal{I}'} \phi(\mathrm{d} t) =0 \right\}. \]
A sequence of probability measures $P_n$ on a Polish space is said to weakly
converge to a probability measure $P$ if for all bounded and continuous functions $h$ we have
\[ P_n h \to P h. \]
Where we write $Ph = \int h(x) \; P(\mathrm{d} x)$.
If $P_n$ weakly converges to $P$ then we write $P_n\Rightarrow P$. 
\subsection{The G\^ateaux Derivative \label{sec:Prelim:GatDer}}
\begin{mydef} \label{def:Prelim:GatDer:GatDer}
We say that $f:\mathcal{H}\to \mathbb{R}$ is G\^ateaux differentiable at $\mu\in \mathcal{H}$ in direction $\nu\in \mathcal{H}$ if the limit
\[ \partial f(\mu;\nu) = \lim_{r\to 0} \frac{f(\mu+r\nu) - f(\mu)}{r} \]
exists.
We may define second order derivatives by
\[ \partial^2 f(\mu;\nu,\nu') = \lim_{r\to 0} \frac{\partial f(\mu+r\nu';\nu) - \partial f(\mu;\nu)}{r} \]
for $\mu,\nu,\nu'\in \mathcal{H}$.
Similarly for higher order derivatives.
To simplify notation, when it is clear, we write
\[ \partial^s f(\mu;\nu):= \partial^s f(\mu;\nu,\dots, \nu). \]
\end{mydef}

\begin{theorem}[Taylor's Theorem]
\label{thm:Prelim:GatDer:Taylor}
If $f:\mathcal{H}\to \mathbb{R}$ is $m$ times continuously G\^ateaux
differentiable on a convex subset $K\subset \mathcal{H}$ then, for $\mu,\nu\in K$:
\begin{align*}
f(\nu) & = f(\mu) + \partial f(\mu;\nu-\mu) + \frac{1}{2!} \partial^2 f(\mu; \nu-\mu, \nu-\mu) + \dots \\
 & + \frac{1}{(m-1)!} \partial^{m-1} f(\mu; \nu-\mu,\dots, \nu-\mu) + R_m
\end{align*}
where
\[ R_m(\mu,\nu-\mu) = \frac{1}{(m-1)!} \int_0^1 (1-t)^{m-1} \partial^m f((1-t)\mu + t\nu;\nu-\mu) \; \mathrm{d} t. \]
\end{theorem}
\subsection{\texorpdfstring{$\Gamma$}{Gamma}-Convergence \label{subsec:Prelim:Gamma}}
Variational methods, and in particular $\Gamma$-convergence, have been used by the authors previously to prove consistency of estimators which arise as solutions to a variational problem~\cite{thorpe15,thorpe16}.
We have the following definition of $\Gamma$-convergence with respect to weak convergence.
\begin{mydef}[$\Gamma$-convergence {\cite[Definition 1.5]{braides02}}]
\label{def:Prelim:Gamma:Gamcon}
Let $\mathcal{H}$ be a Banach space.
A sequence $f_n :\mathcal{H}\to \mathbb{R}\cup\{\pm\infty\}$ is said to \textit{$\Gamma$-converge} on the domain $\mathcal{H}$ to $f_\infty :\mathcal{H}\to \mathbb{R}\cup\{\pm\infty\}$ with respect to weak convergence on $\mathcal{H}$, and we write $f_\infty = \Gamma\text{-}\lim_n f_n$, if for all $\nu\in \mathcal{H}$ we have
\begin{itemize}
\item[(i)] (lim inf inequality) for every sequence $(\nu^n)$ weakly converging to $\nu$
\[ f_\infty(\nu) \leq \liminf_{n\to \infty} f_n(\nu^n); \]
\item[(ii)] (recovery sequence) there exists a sequence $(\nu^n)$ weakly converging to $\nu$ such that
\[ f_\infty(\nu) \geq \limsup_{n\to \infty} f_n(\nu^n). \]
\end{itemize}
\end{mydef}

When it exists the $\Gamma$-limit is always weakly lower semi-continuous \cite[Proposition 1.31]{braides02} and therefore the minimum of the $\Gamma$-limit over weakly compact sets is achieved.
An important property of $\Gamma$-convergence is that it implies the convergence of
almost minimizers where $\mu^n$ is a sequence of almost minimizers of $f_n$ if there exists a sequence $\delta_n$ with $\delta_n\to 0$ and $f_n(\mu^n) \leq \inf f_n + \delta_n$.
In particular, we will make use of the following well known result which can be found in \cite[Theorem 1.21]{braides02}.

\begin{theorem}[Convergence of Minimizers]
\label{thm:Prelim:Gamma:Conmin}
Let $f_n: \mathcal{H}\to \mathbb{R}\cup\{\pm\infty\}$ be a sequence of functionals on a Banach space $(\mathcal{H},\|\cdot\|)$.
Assume there exists a weakly compact subset $K\subset \mathcal{H}$ with 
\[ \inf_{\mathcal{H}} f_n = \inf_K f_n \quad \forall n\in\mathbb{N}. \]
If $f_\infty = \Gamma\text{-}\lim_n f_n$ and $f_\infty$ is not identically $\pm\infty$ then
\[ \min_{\mathcal{H}} f_\infty = \lim_{n\to \infty} \inf_{\mathcal{H}} f_n. \]
Furthermore if $\mu^n\in K$ are almost minimizers of $f_n$ then any weak limit point minimizes $f_\infty$.
\end{theorem}

A simple consequence of the above is the following corollary which avoids
recourse to subsequences.

\begin{corollary}
\label{cor:Prelim:Gamma:Conmin}
If in addition to the assumptions of Theorem~\ref{thm:Prelim:Gamma:Conmin} the minimizer of the $\Gamma$-limit is unique then any sequence of almost minimizers $\mu^n$ of $f_n$ converges weakly to the minimizer of $f_\infty$. 
\end{corollary}
\subsection{The Spline Framework \label{subsec:Prelim:Splines}}
In this subsection we recap the spline methodology and find an explicit representation for our estimators.
In particular we construct our estimate as a minimizer of a quadratic functional.
We will show the existence and uniqueness of the minimizer.

We consider the separable Hilbert space $\mathcal{H}$ with inner product and norm given by $(\cdot,\cdot)$ and $\|\cdot\|$ respectively.
We assume we can write $\mathcal{H}=\mathcal{H}_0\oplus\mathcal{H}_1$ where $(\mathcal{H}_0,( \cdot,\cdot )_0,\|\cdot\|_0)$, $(\mathcal{H}_1,( \cdot,\cdot )_1,\|\cdot\|_1)$ are Hilbert spaces with $\text{dim}(\mathcal{H}_0)=m$ and $\text{dim}(\mathcal{H}_1)=\infty$.
We may write
\[ \|\mu\| = \|\mu\|_0 + \|\mu\|_1. \]
It is convenient to extend the domain of $||\cdot||_i$ from $\mathcal{H}_i$ to $\mathcal{H}$, setting $\|\mu\|_i := \|\chi_i \mu\| = \|\chi_i \mu\|_i$ as  $\mathcal{H}_0$ is orthogonal to $\mathcal{H}_1$ by assumption.
For example, in the special spline case, $\mathcal{H}_0$ is the space of polynomials of degree at most $m-1$ and $\mathcal{H}_1$ will be the space of remainder terms
\[ R(t) = \mu(t) - \sum_{i=0}^{m-1} \frac{\nabla^i \mu(0)}{i!} t^i. \]
The norm on $\mathcal{H}_1$ is $\|\mu\|_1 = \|\nabla^m \mu\|_{L^2}$.
Now the projection of a function $\mu\in\mathcal{H}$ to $\mathcal{H}_1$ is just the projection $\mu\mapsto R$ given by the above expression.
Clearly $\|\mu\|_1 = \|R\|_1 = \|\chi_1 \mu\|_1$.
Since $\mathcal{H}_0$ is finite dimensional we are free to choose the norm without changing the topology, however it is convenient to choose a norm that is orthogonal to $\mathcal{H}_1$ when viewed as a function of $\mathcal{H}$.
A natural choice is $\|\mu\|_0^2 = \sum_{i=0}^{m-1} |\nabla^i \mu(0)|^2$.
The special spline problem is discussed more below, particularly in Section~\ref{sec:Ex}.

We wish to estimate $\mu^\dagger\in\mathcal{H}$ given observations of the form $(L_i,y_i)$
and $L_i$ (as well as $y_i$) is random.
For convenience we summarize the general spline model in the definition below.
One can also see, for example,~\cite{wahba90} for more details on the general spline model.

\begin{mymodel*}
\label{def:Prelim:Splines:Model}
The general spline model is given by~\eqref{eq:Intro:Data} where $L_i\in\mathcal{H}^*$ are random variables and $\epsilon_i$ are iid random variables from a centered distribution, $\phi_0$, with variance $\sigma^2$.
The $L_i$ are assumed to be observed without noise and to be members of a family indexed by $\mathcal{I}\subset\mathbb{R}^d$; we write $L_t$ to mean the operator $L$ which depends upon a parameter $t\in\mathcal{I}$.
The `randomness' of $L$ is characterized by the distribution, $\phi_T$, of a
random index $t\in\mathcal{I}$.
For a sample $t_i\sim \phi_T$ we write $L_i$ as shorthand for $L_{t_i}$.
The operator $L_i$ is therefore interpreted as a realization of $L_{t_i}$.
We assume that $t_i,\epsilon_i$ are independent and for convenience we define $\phi_{L_t\mu^\dagger}$ to be the distribution $\phi_0$ shifted by $-L_t\mu^\dagger$.
By the Riesz Representation Theorem there exists $\eta_i\in\mathcal{H}$ such that $L_i\mu=(\eta_i,\mu)$ for all $\mu\in \mathcal{H}$.
The sequence of observed data points $(t_1,y_1),(t_2,y_2),\dots$  is a realization of a sequence of random elements on $(\Omega,\mathcal{F},\mathbb{P})$.
To mitigate the notational burden, we suppress the $\omega$-dependence of $t_i$, $y_i$ and $L_i$.
\end{mymodel*}


For example in the case of special splines $L_i \mu^\dagger = \mu^\dagger(t_i)$ for some $t_i$ a random variable distributed in $[0,1]$.
Observing $L_i$ without noise is equivalent here to observing $t_i$ without noise.
We refer to Section~\ref{sec:Ex} for more details.

We take our sequence of estimators $\mu^n$ of $\mu^\dagger$ as minimizers,
which are subsequently shown to be unique,  of $f_n^{(\omega)}$ where
\begin{equation} \label{eq:Prelim:Splines:fn}
f_n^{(\omega)}(\mu) = \frac{1}{n} \sum_{i=1}^n (y_i-L_i\mu)^2 + \lambda_n \|\mu\|^2_1.
\end{equation}
By completing the square we can easily show $\mu^n$ is given implicitly by
\[ G_{n,\lambda_n} \mu^n = \frac{1}{n} \sum_{i=1}^n y_i\eta_i \]
where
\begin{equation} \label{eq:Prelim:Splines:Gnlambda}
G_{n,\lambda} = \frac{1}{n}\sum_{i=1}^n \eta_iL_i + \lambda \chi_1
\end{equation}
and for clarity we also suppress the $\omega$-dependence of $G_{n,\lambda}$ from the notation.
It will be necessary in our proofs to bound $\|G_{n,\lambda_n}\|_{\mathcal{H}^*}$ in terms of $\lambda_n$ (for almost every sequence of observations).
We do this by imposing a bound on $\|L_t\|_{\mathcal{H}^*}$ or equivalently on $\|\eta_t\|$ for almost every $t\in\mathcal{I}$.
See Section~\ref{sec:Ex} for a discussion of the special spline problem and in particular how one can find $\eta_i$.
In order to bound the $\mathcal{H}_0$ norm of $\mu^n$ we need conditions on our observation operators $L_t$.
In particular we will use the observation operators to define a norm on $\mathcal{H}_0$.
Hence our proofs require a uniqueness assumption of $L_t$ in $\mathcal{H}_0$ (Assumption~\ref{ass:Prelim:Spline:Lunique} below).
It is not enough that $L_t$ are unique over $\mathcal{H}$ as this would not necessarily contain any information on the $\mathcal{H}_0$ projection of $\mu^n$, e.g. if $L_t\mu = L_t \chi_1 \mu$ for all $\mu\in\mathcal{H}$.
For clarity and future reference we now summarize the assumptions described in the previous paragraphs.

\begin{assumptions}
\label{mod:Prelim:Splines:Model}
We make the following assumptions on $f_n^{(\omega)}:\mathcal{H}\to \mathbb{R}$ defined by~\eqref{eq:Prelim:Splines:fn} and $\mathcal{H}$.
\begin{enumerate}
\item \label{ass:Prelim:Spline:Space} Let $(\mathcal{H},(\cdot,\cdot),\|\cdot\|)$ be a separable Hilbert space with $\mathcal{H}=\mathcal{H}_0\oplus\mathcal{H}_1$ where $(\mathcal{H}_0,(\cdot,\cdot)_0,\|\cdot\|_0)$ and $(\mathcal{H}_1,(\cdot,\cdot)_1,\|\cdot\|_1)$ are Hilbert spaces. Assume $\text{dim}(\mathcal{H})=\text{dim}(\mathcal{H}_1)=\infty$ and $\text{dim}(\mathcal{H}_0)=m<\infty$. 
\item \label{ass:Prelim:Spline:phimudagger} The distribution of $L_i:=L_{t_i}$
  is specified implicitly by that of $t_i \in \mathcal{I}\subset\mathbb{R}^d$ and we assume $t_i \iid \phi_T$.
\item \label{ass:Prelim:Spline:Lunique} We assume $|\text{supp}(\phi_T)| \geq m$ and
  that the $L_t$ are unique in $\mathcal{H}_0$ in the sense that if $L_t \mu=L_r \mu$ for all $\mu\in\mathcal{H}_0$ then $t=r$.
\item \label{ass:PreLim:Splines:etabound} There exists $\alpha>0$ such that $\|\eta_t\| = \|L_t\|_{\mathcal{H}^*} \leq \alpha$ for $\phi_T$-almost every $t\in\mathcal{I}$.
\end{enumerate}
\end{assumptions}

For the general spline problem we allow multivariate regression, that is $t_i\in \mathbb{R}^d$, see for example~\cite[Section 7]{xiao13} for multivariate P-splines.
However, when discussing the special spline problem we will often assume $d=1$ since, although our convergence results still hold for $d>1$, there are regularity issues such as that for $2m<d$ minimizers are not automatically continuous (for $2m>d$ the Sobolev space $H^m$ on $\mathbb{R}^d$ is embedded in $C^0$, this is not true for $2m<d$). 

The existence of a unique minimizer to \eqref{eq:Prelim:Splines:fn} is
established in the following lemma.
\begin{lemma}
\label{lem:Prelim:Splines:Minfn}
Define $f_n^{(\omega)}:\mathcal{H}\to\mathbb{R}$ by~\eqref{eq:Prelim:Splines:fn} and assume $\lambda_n>0$.
Under Assumptions \ref{ass:Prelim:Spline:Space}-\ref{ass:PreLim:Splines:etabound} the operator $G_{n,\lambda_n}:\mathcal{H}\to\mathcal{H}$ defined by~\eqref{eq:Prelim:Splines:Gnlambda} has a well defined inverse $G_{n,\lambda_n}^{-1}$ on $\text{span}\{\eta_1,\dots,\eta_n\}$ for almost every $\omega\in \Omega$.
In particular, there almost surely exists $N<\infty$ such that for all $n\geq N$ there exists a unique minimizer $\mu^n\in\mathcal{H}$ to $f_n^{(\omega)}$ which is given by
\begin{equation} \label{eq:Prelim:Splines:mun}
\mu^n = \frac{1}{n} \sum_{i=1}^n y_i G_{n,\lambda_n}^{-1} \eta_i.
\end{equation}
\end{lemma}

\begin{proof}
We claim that any minimizer of $f_n^{(\omega)}$ lies in the set $\mathcal{H}_0 \oplus \text{span}\{\chi_1 \eta_1,\dots,\chi_1 \eta_n\} =:\mathcal{H}_n^\prime$.
If  this is so, and it can be shown that $G_{n,\lambda_n}^{-1}$ is well defined on $\mathcal{H}_n^\prime$, then we can conclude the minimizer must be of the form~\eqref{eq:Prelim:Splines:mun}.

We define $\Omega^\prime\subset \Omega$ by
\[ \Omega^\prime := \left\{ \omega\in \Omega : \text{the number of unique } t_j \text{ in } \{t_i\}_{i=1}^\infty \text{ is greater than } m \text{ and } \|L_i\|_{\mathcal{H}^*} \leq \alpha \, \forall i \right\}. \]
By Assumptions~\ref{ass:Prelim:Spline:Lunique} and~\ref{ass:PreLim:Splines:etabound}, $\mathbb{P}(\Omega^\prime)=1$.
Let $\omega\in \Omega^\prime$ then there exists $N$ such that for all $n\geq N$ we have that $\{L_i\}_{i=1}^N$ contains $m$ distinct elements. 
Therefore $\|\mu\|^2_{\mathcal{H}_n^\prime} := \frac{1}{n} \sum_{i=1}^n (L_i\mu)^2 + \lambda_n
\|\mu\|_1^2$ defines a norm on $\mathcal{H}_n^\prime$ for any $n\geq N$ and, as $\mathcal{H}_n^\prime$ is finite dimensional, we arrive at the same topology whichever norm we choose.

We first show that any minimizer of $f_n^{(\omega)}$ lies in $\mathcal{H}_n^\prime$.
Let $\mu = \sum_{j=1}^m a_j \phi_j + \sum_{j=1}^n b_j \chi_1 \eta_j + \rho$ where $\phi_j$ are a basis for $\mathcal{H}_0$ and $\rho \perp \mathcal{H}_n^\prime$.
Then since $L_i\rho = (\eta_i,\rho) = 0$ we have:
\[ f_n^{(\omega)}(\mu) = \frac{1}{n} \sum_{i=1}^n \left( y_i - L_i \chi_{\mathcal{H}_n^\prime}
  \mu\right)^2 + \lambda_n \left\|\sum_{j=1}^n b_j \chi_1 \eta_j \right\|_1^2 + \lambda_n \|\rho\|_1^2 \]
where $\chi_{\mathcal{H}_n^\prime}$ denotes the projection onto $\mathcal{H}_n^\prime$.
Trivially any minimizer of $f_n^{(\omega)}$ must have $\|\rho\|_1=0$ and since $\rho\in\mathcal{H}_1$ this implies $\rho=0$.
Hence minimizers of $f_n^{(\omega)}$ lie in $\mathcal{H}_n^\prime$.

We now show that $G_{n,\lambda_n}$ has a well defined inverse on $\mathcal{H}_n^\prime$; that is we want to show that for any $r\in \mathcal{H}_n^\prime$ there exists $\mu^n\in \mathcal{H}_n^\prime$ such that $G_{n,\lambda_n} \mu^n = r$. 
The weak formulation of $G_{n,\lambda_n}\mu^n=r$ is given by
\[ B(\mu^n,\nu) = (r,\nu) \quad \quad \forall \nu \in \mathcal{H}_n^\prime \]
where
\[ B(\mu,\nu) = \frac{1}{n} \sum_{i=1}^n (L_i\mu)(L_i\nu) + \lambda_n(\chi_1 \mu, \chi_1 \nu). \]
Now we apply the Lax-Milgram lemma to imply there exists a unique weak solution. 
Clearly $B:\mathcal{H}_n^\prime\times\mathcal{H}_n^\prime\to\mathbb{R}$ is a bilinear form.
We will show it is also bounded and coercive.
As $\omega\in\Omega^\prime$, $\|L_i\|_{\mathcal{H}^*}\leq \alpha$ and for $\mu,\nu\in\mathcal{H}_n^\prime$ we have 
\begin{align*}
| B(\mu,\nu)| & \leq \frac{1}{n} \sum_{i=1}^n |L_i\mu L_i\nu | + \lambda_n \|\mu\|_1 \|\nu\|_1 \\
 & \leq \alpha^2 \|\mu\| \|\nu\| + \lambda_n \|\mu\|_1 \|\nu\|_1 \\
 & \leq \left(\alpha^2 + \lambda_n\right) \|\mu\| \|\nu\|.
\end{align*}
Hence $B$ is bounded.
Similarly, for some constant $c$ independent of $\mu$,
\[ B(\mu,\mu) = \frac{1}{n} \sum_{i=1}^n (L_i\mu)^2 + \lambda_n \|\mu\|_1^2 = \|\mu\|_{\mathcal{H}_n^\prime}^2 \geq c \|\mu\|^2 \]
where the inequality follows by the equivalence of norms on finite dimensional spaces.
%
Hence $B$ is coercive and by the Lax-Milgram Lemma there exists a unique weak solution.
We have shown that for any $r\in \mathcal{H}_n^\prime$ there exists $\mu_n\in\mathcal{H}_n^\prime$ such that $B(\mu^n,\nu)=(r,\nu)$ for all $\nu\in \mathcal{H}_n^\prime$.

A strong solution follows from the equivalence of the strong and weak topology on finite dimensional spaces or alternatively from the following short calculation.
We have
\[ (r,\nu) = \left( \frac{1}{n} \sum_{i=1}^n (L_i\mu^n)\eta_i,\nu \right) + \left( \lambda_n \chi_1 \mu^n,\nu\right) \quad \quad \forall \nu\in\mathcal{H}_n^\prime. \]
Hence
\[ \left( r - \frac{1}{n} \sum_{i=1}^n (L_i\mu^n)\eta_i - \lambda_n \chi_1 \mu^n,\nu\right) = 0 \quad \quad \forall \nu\in\mathcal{H}_n^\prime. \]
So choosing $\nu = r - \frac{1}{n} \sum_{i=1}^n (L_i\mu^n)\eta_i - \lambda_n \chi_1 \mu^n$ implies $\| r - \frac{1}{n} \sum_{i=1}^n (L_i\mu^n)\eta_i - \lambda_n \chi_1 \mu^n \|^2 =0$ and therefore
\[ r = \frac{1}{n} \sum_{i=1}^n (L_i\mu^n)\eta_i - \lambda_n \chi_1 \mu^n = G_{n,\lambda_n}\mu^n. \]
As this is true for all $r\in\mathcal{H}_n^\prime$ we can infer the existence of an inverse operator $G_{n,\lambda_n}^{-1}:\mathcal{H}_n^\prime\to\mathcal{H}_n^\prime$ such that $G_{n,\lambda_n}^{-1}r=\mu^n$.
One can verify that $G_{n,\lambda_n}^{-1}$ is linear. As $\omega \in \Omega^\prime$ was
arbitrary, the result holds almost surely.
\end{proof}
\section{Consistency \label{sec:Cons}}
We demonstrate consistency by applying the $\Gamma$-convergence framework.
This requires us to find the $\Gamma$-limit, to show that the $\Gamma$-limit has a unique minimizer and that the minimizers of $f_n^{(\omega)}$ are uniformly bounded.
The next three subsections demonstrate that each of these requirements is
satisfied under the stated assumptions and allow the application of Corollary~\ref{cor:Prelim:Gamma:Conmin} to conclude the consistency of the spline model, as summarized in Theorem~\ref{thm:Cons:Cons}.
We start by stating the remainder of the conditions employed.

\begin{assumptions}
\begin{enumerate}
\setcounter{enumi}{4}
\item \label{ass:Cons:lambdascale} We have $\lambda_n=n^{-p}$ with $0<p\leq\frac{1}{2}$.
\item \label{ass:Cons:Norm} For $\nu\in\mathcal{H}$ the following relation holds:
\[ \int_{\mathcal{I}} (L_t\nu)^2 \phi_T(\mathrm{d}t) = 0 \Leftrightarrow \nu=0. \]
\item \label{ass:Cons:Cts} For each $\mu\in\mathcal{H}$ each $L_t\mu$ is continuous in $t$, i.e $\|L_s-L_t\|_{\mathcal{H}^*} \to 0$ as $s\to t$.
\end{enumerate}
\end{assumptions}

Assumption~\ref{ass:Cons:lambdascale} gives the admissible scaling regime in $\lambda_n$.
Clearly if $p\leq 0$ then $\lambda_n\not\to 0$ hence we expect the limit, if it even
exists, to be biased towards solutions more regular than $\mu^\dagger$.
We are required to show that the minimizers are bounded in probability.
To do so we show they are bounded in expectation.
We will show in Theorem~\ref{thm:Cons:Sharp} that for $p>\frac{1}{2}$ we cannot bound minimizers in expectation; hence it is not possible to extend our proofs for $p\not\in (0,\frac{1}{2}]$.
Theorem~\ref{thm:Cons:Cons} holds as it does and not in expectation because the $\Gamma$-convergence framework requires $\mu^n$ to be a minimizer and as such we cannot make conclusions about the ``average minimizer'' since $\mathbb{E}[\mu^n|\mathcal{G}_n]$ is not a minimizer.

We will show that the second derivative of $f_\infty$ in the direction $\nu$ is given by $\int_{\mathcal{I}} (L_t\nu)^2 \; \phi_T(\mathrm{d}t)$.
Assumption~\ref{ass:Cons:Norm} is used to establish that $f_\infty$ is strictly convex, and hence the minimizer is unique.

It will be necessary to show that
\begin{equation} \label{eq:Cons:IntLtmu}
\frac{1}{n} \sum_{i=1}^n |L_i\mu| \to \int_{\mathcal{I}} |L_t\mu| \; \phi_T(\mathrm{d}t)
\end{equation}
for all $\mu\in\mathcal{H}$ with probability one.
We impose Assumption~\ref{ass:Cons:Cts} (together with Assumption~\ref{ass:PreLim:Splines:etabound}) to imply that $L_t\mu$ is continuous and bounded in $t$ for all $\mu\in\mathcal{H}$ and therefore by the weak convergence of the empirical measure we infer that~\eqref{eq:Cons:IntLtmu} holds for all $\mu\in\mathcal{H}$ and for almost every sequence $\{L_i\}_{i=1}^\infty$.
In particular we can define a set $\Omega^\prime\subset\Omega$ independent of $\mu$, on
which~\eqref{eq:Cons:IntLtmu} holds, such that $\mathbb{P}(\Omega^\prime)=1$.

\begin{theorem}
\label{thm:Cons:Cons}
Define $f_n^{(\omega)}:\mathcal{H}\to\mathbb{R}$ by~\eqref{eq:Prelim:Splines:fn}.
Under Assumptions~\ref{ass:Prelim:Spline:Space}-\ref{ass:Cons:Cts} the minimizer $\mu^{n}$ of $f_n^{(\omega)}$ converges in the following sense: 
for all $\epsilon,\delta>0$ and $F\in \mathcal{H}^*$ there exists $N = N(\epsilon,\delta,F)\in\mathbb{N}$ such that
\[ \mathbb{P}\left( \left|F(\mu^{n})-F(\mu^\dagger)\right| \geq \epsilon \right) \leq \delta \quad \text{for } n\geq N. \]
\end{theorem}

\begin{remark}
\label{rem:Cons:ConvNot}
We view the mode of convergence in the above theorem as a natural generalization of
convergence in probability; it is weaker than convergence \emph{weakly in
  probability}, which would require that the convergence of $\mu^{n}\to \mu^\dagger$ were
uniform over $F\in\mathcal{H}^*$ and not pointwise as established in the theorem. 
\end{remark}

The following theorem shows that if $p>\frac{1}{2}$ then without imposing further assumptions it is always possible to construct observation functionals $\{L_t\}_{t\in\mathcal{I}}$ such that $\mathbb{E}\left[ \|\mu^n\|^2\right]\to\infty$. 

\begin{theorem}
\label{thm:Cons:Sharp}
Define $f_n^{(\omega)}:\mathcal{H}\to\mathbb{R}$ by~\eqref{eq:Prelim:Splines:fn}, let $\mu^n$ be the minimizer of $f_n^{(\omega)}$ and take any $\alpha>0$ and $p>\frac{1}{2}$.
Take Assumptions \ref{ass:Prelim:Spline:Space}-\ref{ass:Prelim:Spline:phimudagger} and assume that $\lambda=n^{-p}$.
Then there exists a distribution $\phi_T$ on $\mathcal{I}$ such that
$\|L_t\|_{\mathcal{H}^*} = \|\eta_t\| \leq \alpha$ for almost every $\omega\in\Omega$
(i.e. Assumption~\ref{ass:PreLim:Splines:etabound} holds) and $\mathbb{E}[\|\mu^n\|^2]\to \infty$.
\end{theorem}

In the special spline model, when $\lambda\to 0$ too quickly the functions $\mu^n$ begin to interpolate the data points $\{(t_i,y_i)\}_{i=1}^n$, hence the derivative of $\mu^n$ will not stay bounded.
Furthermore, when considering weak convergence, one is restricting to finite dimensional projections.
It is therefore not surprising that $n^{-\frac{1}{2}}$ is the best we can do.
For $p>\frac{1}{2}$ and a sequence of real valued iid random variables $X_i$ of
finite variance (which are not identically zero) we have $n^{2p}\mathbb{E}(\frac{1}{n}\sum_{i=1}^n X_i)^2 \to \infty$.
In light of this elementary observation Theorem~\ref{thm:Cons:Sharp} is not
surprising. The proof is given in Section~\ref{subsec:Cons:Sharp}.
\subsection{The \texorpdfstring{$\Gamma$}{Gamma}-Limit \label{subsec:Cons:Gamma}}
We claim the $\Gamma$-limit of $f_n^{(\omega)}$, for almost every $\omega\in\Omega$, is given by
\begin{equation} \label{eq:Cons:Gamma:GamLim}
f_\infty(\mu) = \int_{\mathcal{I}} \int_{-\infty}^\infty |y-L_t \mu|^2 \; \phi_{L_t\mu^\dagger}(\mathrm{d} y) \; \phi_T(\text{d} t).
\end{equation}

\begin{theorem} \label{thm:Cons:Gamma:GammaConv}
Define $f_n^{(\omega)},f_\infty:\mathcal{H}\to \mathbb{R}$ by~\eqref{eq:Prelim:Splines:fn} and~\eqref{eq:Cons:Gamma:GamLim} respectively.
Under Assumptions~\ref{ass:Prelim:Spline:Space}-\ref{ass:Prelim:Spline:phimudagger}, \ref{ass:Cons:lambdascale} and~\ref{ass:Cons:Cts},
\[ f_\infty = \Gamma\text{-}\lim_n f_n^{(\omega)} \]
for almost every $\omega\in\Omega$.
\end{theorem}

\begin{proof}
We are required to show the two inequalities in
Definition~\ref{def:Prelim:Gamma:Gamcon} hold with probability 1.
In order to do this we consider a subset of $\Omega$ of full measure, $\Omega^\prime$, and show that both statements hold for every data sequence obtained from that set.

Define $g_\mu(t,y) = (y-L_t\mu)^2$.
For clarity let $P(\mathrm{d}(t,y)) = \phi_T(\mathrm{d}t)\phi_{L_t\mu^\dagger}(\mathrm{d}y)$ and
$P_n$ be the empirical measure associated with the observations, i.e. for any measurable  $h:\mathcal{I}\times\mathbb{R}\to \mathbb{R}$ we define $P_n h = \frac{1}{n} \sum_{i=1}^n h(t_i,y_i)$.
Further, let $P_n^{(\omega)}$ denote the measure arising from the particular realization $\omega$.
Defining:
\[ \Omega^\prime = \left\{\omega: P_n^{(\omega)} \Rightarrow P\right\} \cap \left\{ \omega\in\Omega : \frac{1}{n}\sum_{i=1}^n\epsilon_i^2(\omega)\to \sigma^2 \text{ and } \frac{1}{n}\sum_{i=1}^n\epsilon_i(\omega)\to 0 \right\} , \]
then $\mathbb{P}(\Omega^\prime)=1$ by the almost sure weak convergence of the
empirical measure~\cite[Theorem 11.4.1]{dudley02} and the strong law of large numbers.
Let $\omega\in \Omega^\prime$.

We start with the lim inf inequality.
Pick $\nu\in\mathcal{H}$ and let $\nu^n\rightharpoonup \nu$.
By Theorem~1.1 in~\cite{feinberg14} we have
\begin{align*}
\int_{\mathcal{I}} \int_{-\infty}^\infty \liminf_{n\to\infty,(t^\prime,y^\prime)\to (t,y)} g_{\nu^n}(t^\prime,y^\prime) \; P(\mathrm{d}(t,y)) & \leq \liminf_{n\to\infty} \int_{\mathcal{I}} \int_{-\infty}^\infty g_{\nu^n}(t,y)\; P_n^{(\omega)}(\mathrm{d}(t,y)) \\
 & = \liminf_{n\to \infty} f_n^{(\omega)}(\nu^n).
\end{align*}
Now we show
\begin{equation} \label{eq:Cons:Gamma:lsc}
\liminf_{n\to\infty,(t^\prime,y^\prime)\to (t,y)} g_{\nu^n}(t^\prime,y^\prime) \geq g_\nu(t,y)
\end{equation}
which proves the lim inf inequality.
Let $(t_m,y_m)\to (t,y)$ then
\begin{align*}
\left( g_{\nu^n}(t_m,y_m) \right)^{\frac{1}{2}} & = |y_m - L_{t_m}\nu^n| \\
 & \geq |L_{t_m}\nu^n - y| - |y_m-y| \\
 & \geq |y-L_t\nu^n| - | L_{t_m}\nu^n - L_t\nu^n| - |y_m-y| \\
 & \geq |y-L_t\nu^n| - \| L_{t_m} - L_t \|_{\mathcal{H}^*} \|\nu^n\| - |y_m-y|.
\end{align*} 
A consequence of the uniform boundedness principle is that any weakly convergent sequence is bounded, hence there exists some $C>0$ such that $\|\nu^n\|\leq C$.
It follows from the above, and Assumption \ref{ass:Cons:Cts}, that
\[ \liminf_{n\to\infty,m\to \infty} \left( g_{\nu^n}(t_m,y_m) \right)^{\frac{1}{2}} \geq |y-L_t\nu| = \left(g_\nu(t,y)\right)^{\frac{1}{2}}. \]
As our choice of sequence $(t_m,y_m)$ was arbitrary we can conclude that~\eqref{eq:Cons:Gamma:lsc} holds.

For the recovery sequence we choose $\nu\in\mathcal{H}$ and let $\nu^n=\nu$.
We are required to show
\[ Pg_\nu \geq \limsup_{n\to \infty} \left( P_n^{(\omega)}g_\nu + \lambda_n \|\mu\|_1^2 \right) = \limsup_{n\to \infty} P_n^{(\omega)}g_\nu. \]
Since we can write
\[ g_\nu(t_i,y_i) = (L_i \mu^\dagger)^2 + \epsilon_i^2 + (L_i\nu)^2 + 2\epsilon_i L_i\mu^\dagger - 2L_i\mu^\dagger L_i\nu - 2\epsilon_i L_i\nu \]
and each term is either a continuous and bounded functional, or its convergence
is addressed directly by the construction of $\Omega^\prime$,  we have
$P_n^{(\omega)}g_\nu \to Pg_\nu$ as required. As $\omega \in \Omega^\prime$ was
arbitrary, the result holds almost surely.
\end{proof}

\begin{remark}
\label{rem:Cons:Gamma:lambascale}
Note that in the above theorem we did not need a lower bound on the decay of $\lambda_n$ (only that $\lambda_n\geq 0$).
We only used that $\lambda_n=o(1)$.
\end{remark}
\subsection{Uniqueness of the \texorpdfstring{$\Gamma$}{Gamma}-limit \label{subsec:Cons:Unique}}
To show the $\Gamma$-limit has a unique minimizer we show it is strictly convex.
The following lemma gives the second G\^ateaux derivative of $f_\infty$.
After which we conclude in Corollary~\ref{cor:Cons:Unique:Unique} that the $\Gamma$-limit is unique.

\begin{lemma}
\label{lem:Cons:Unique:2ndDer}
Under
Assumptions~\ref{ass:Prelim:Spline:Space}-\ref{ass:Prelim:Spline:phimudagger} define
$f_\infty:\mathcal{H}\to\mathbb{R}$ by~\eqref{eq:Cons:Gamma:GamLim}.
Then the first and second G\^ateaux derivatives of $f_\infty$ are given by
\begin{align*}
\partial f_\infty(\mu;\nu) & = 2 \int_{\mathcal{I}} \int_{-\infty}^\infty (L_t\mu-y) L_t(\nu) \phi_{L_t\mu^\dagger}(\mathrm{d} y) \phi_T(\mathrm{d} t) \\
\partial^2 f_\infty(\mu;\nu,\zeta) & = 2 \int_\mathcal{I} (L_t \nu) (L_t\zeta) \phi_T(\mathrm{d} t).
\end{align*}
\end{lemma}

\begin{proof}
We first compute the first G\^ateaux derivative.
We have
\begin{align*}
\partial f_\infty(\mu; \nu) & = \lim_{r\to 0} \int_{\mathcal{I}} \int_{-\infty}^\infty \frac{(y-L_t(\mu+r\nu))^2 - (y-L_t\mu)^2}{r} \; \phi_{L_t\mu^\dagger}(\mathrm{d}y) \phi_T(\mathrm{d}t) \\
 & = 2 \int_{\mathcal{I}} \int_{-\infty}^\infty (L_t\mu-y) L_t(\nu) \phi_{L_t\mu^\dagger}(\mathrm{d} y) \phi_T(\mathrm{d} t) + \lim_{r\to 0} r \int_{\mathcal{I}} \int_{-\infty}^\infty (L_t\nu)^2 \phi_{L_t\mu^\dagger}(\mathrm{d}y) \phi_T(\mathrm{d} t) \\
  & = 2 \int_{\mathcal{I}} \int_{-\infty}^\infty (L_t\mu-y) L_t(\nu) \phi_{L_t\mu^\dagger}(\mathrm{d} y) \phi_T(\mathrm{d} t) \quad \text{recalling that } L_t \text{ is linear.}
\end{align*}
The second G\^ateaux derivative follows similarly.
\begin{align*}
\partial^2 f_\infty(\mu;\nu,\zeta) & = \lim_{r\to 0} 2 \int_{\mathcal{I}} \int_{-\infty}^\infty \frac{(L_t(\mu+r\zeta)-y)L_t\nu - (L_t\mu-y)L_t\nu}{r} \phi_{L_t\mu^\dagger}(\mathrm{d}y) \phi_T(\mathrm{d}t) \\
 & = 2 \int_{\mathcal{I}} \int_{-\infty}^\infty (L_t\nu)(L_t\zeta) \phi_{L_t\mu^\dagger}(\mathrm{d}y)\phi_T(\mathrm{d}t) \\
 & = 2\int_{\mathcal{I}} (L_t\nu)(L_t\zeta) \phi_T(\mathrm{d} t).
\end{align*}
\end{proof}

\begin{corollary}
\label{cor:Cons:Unique:Unique}
Under Assumptions~\ref{ass:Prelim:Spline:Space}-\ref{ass:Prelim:Spline:phimudagger} and~\ref{ass:Cons:Norm}, define $f_\infty:\mathcal{H}\to\mathbb{R}$ by~\eqref{eq:Cons:Gamma:GamLim}.
Then $f_\infty$ has a unique minimizer which is achieved for $\mu=\mu^\dagger$.
\end{corollary}

\begin{proof}
It is easy to check that $\partial f_\infty(\mu^\dagger;\nu)=0$ for all $\nu\in\mathcal{H}$.
By Lemma~\ref{lem:Cons:Unique:2ndDer} and Assumption~\ref{ass:Cons:Norm} the second G\^ateaux derivative satisfies $\partial^2 f_\infty(\mu;\nu) >0$ for all $\nu\neq 0$.
Then by Taylor's Theorem (and noting that $f_\infty$ is quadratic), for $\mu\neq\mu^\dagger$,
\[ f_\infty(\mu) = f_\infty(\mu^\dagger) + \frac{1}{2} \partial^2 f_\infty(\mu^\dagger;\mu-\mu^\dagger) > f_\infty(\mu^\dagger) \]
as required.
\end{proof}
\subsection{Bound on Minimizers \label{subsec:Cons:Bound}}
In this subsection we show that $\|\mu^n\|=O_p(1)$.
The bound in $\mathcal{H}_0$ can be obtained using fewer assumptions (than the bound in $\mathcal{H}$), which is natural considering $\mathcal{H}_0$ is finite dimensional.
We may choose the norm on $\mathcal{H}_0$ without changing the topology (all norms are equivalent on finite dimensional spaces).
We will use
\[ \|\mu\|_0 = \int_{\mathcal{I}} |L_t \mu| \phi_T(\mathrm{d}t). \]
Loosely speaking we can then write $\|\mu^n\|_0 \lesssim f_n^{(\omega)}(\mu^n)$.
The bound in $\mathcal{H}_0$ then follows if $\min f_n^{(\omega)}$ is bounded.
We make this argument rigorous in Lemma~\ref{lem:Cons:Bound:H0Bound}.
After this result we concentrate on bounding $\mu^n$ in $\mathcal{H}$.

\begin{lemma}
\label{lem:Cons:Bound:H0Bound}
Define $f_n^{(\omega)}:\mathcal{H}\to \mathbb{R}$ by~\eqref{eq:Prelim:Splines:fn}.
Under Assumptions~\ref{ass:Prelim:Spline:Space}-\ref{ass:Cons:lambdascale} and~\ref{ass:Cons:Cts} the minimizers $\mu^n$ of $f_n^{(\omega)}$ are, with probability one, eventually bounded in $\mathcal{H}_0$, i.e. for almost every $\omega\in\Omega$ there exist constants $C,N>0$ such that $\|\mu^n\|_0\leq C$ for all $n\geq N$. 
\end{lemma}

\begin{proof}
We define $P$ and $P_n^{(\omega)}$ as in the proof of Theorem~\ref{thm:Cons:Gamma:GammaConv}, let
\[ \Omega^{\prime} = \left\{ \omega\in\Omega : P_n^{(\omega)} \Rightarrow P \right\} \cap \left\{ \omega\in\Omega : \frac{1}{n}\sum_{i=1}^n \epsilon_i^2(\omega) \to \sigma^2 \text{ and } \frac{1}{n}\sum_{i=1}^n |\epsilon_i(\omega)| \to P|\epsilon_1| \right\} \]
and $\mu^n$ be a minimizer of $f_n^{(\omega)}$.
Assume $\omega\in\Omega^{\prime}$.
As 
\[ f_n^{(\omega)}(\mu^n)\leq f_n^{(\omega)}(\mu^\dagger) \leq \frac{1}{n} \sum_{i=1}^n \epsilon_i^2 + \lambda_1 \|\mu^\dagger\|_1^2
\to \sigma^2 + \lambda_1\|\mu^\dagger\|_1^2,\]  
there exists $N$ such that $f_n^{(\omega)}(\mu^n)\leq \sigma^2 + \lambda_1\|\mu^\dagger\|_1^2 + 1$ for $n\geq N$.

Note that for any $a,b\in \mathbb{R}$ we have
\[ |a - b|^2 \geq \left\{ \begin{array}{ll} |a - b| & \text{if } |a-b| \geq 1 \\ |a-b|-1 & \text{otherwise.} \end{array} \right. \]
In either case $|a-b|^2 \geq |a-b|-1 \geq |a| - |b| - 1$.
Now
\begin{align*}
f_n^{(\omega)}(\mu) & = \frac{1}{n} \sum_{i=1}^n \left(y_i-L_i\mu\right)^2 + \lambda_n \|\mu\|_1^2 \\
 & \geq \frac{1}{n} \sum_{i=1}^n \left(|L_i\mu| - |y_i| -1 \right) \\
 & = \frac{1}{n} \sum_{i=1}^n |L_i\mu| - \frac{1}{n} \sum_{i=1}^n |y_i| - 1 \\
 & \geq \frac{1}{n} \sum_{i=1}^n |L_i\mu| - \frac{1}{n} \sum_{i=1}^n |L_i\mu^\dagger| - \frac{1}{n} \sum_{i=1}^n |\epsilon_i| - 1 \\
 & \to \int_{\mathcal{I}} |L_t \mu| \phi_T(\mathrm{d} t) - c 
\end{align*}
where the convergence follows since $|L_t\mu|$ is a continuous and bounded functional in $t$ and $c$ is given by
\[  \lim_{n\to \infty} \left( \frac{1}{n} \sum_{i=1}^n |L_i\mu^\dagger| + \frac{1}{n} \sum_{i=1}^n
  |\epsilon_i| + 1 \right) \leq \int_{\mathcal{I}} |L_t \mu^\dagger| \phi_T(\mathrm{d} t) + \sigma +1 =: c. \]
We now show that $\int_{\mathcal{I}} |L_t \mu| \phi_T(\mathrm{d} t)$ is a norm on
$\mathcal{H}_0$ and hence that the above constant, $c$, is finite.
This will also show that $\|\mu\|_0 \leq f_n^{(\omega)}(\mu) + c$ for $n \geq N$, which completes the proof.

The triangle inequality, absolute homogeneity and that $\int_{\mathcal{I}} |L_t \mu| \phi_T(\mathrm{d} t)\geq 0$ are trivial to establish.
By
Assumption~\ref{ass:Prelim:Spline:Lunique}, we have at least $m$ disjoint
subsets of positive measure (with respect to $\phi_T$) on $\mathcal{I}$. If
$\int_{\mathcal{I}} |L_t \mu| \phi_T(\mathrm{d} t)=0$ then it follows that on each of these subsets $L_t\mu=0$.
As $\mathcal{H}_0$ is $m$-dimensional this determines $\mu$, and hence $\mu = 0$.

As $\omega \in \Omega^\prime$ was
arbitrary and $\mathbb{P}(\Omega^\prime)=1$, the result holds almost surely.
\end{proof}

\begin{remark}
\label{rem:Cons:Bound:H0bound}
In the above lemma we did not need the lower bound on $\lambda_n$ (only that $\lambda_n\geq 0$).
The result holds for all $\lambda_n=O(1)$.
\end{remark}

Continuing with the bound in $\mathcal{H}$ we write
\begin{equation} \label{eq:Cons:Bound:decompmu}
\mu^n = \frac{1}{n} \sum_{i=1}^n L_i \mu^\dagger G^{-1}_{n,\lambda_n} \eta_i + \frac{1}{n}\sum_{i=1}^n \epsilon_i G_{n,\lambda_n}^{-1}\eta_i = G_{n,\lambda_n}^{-1} U_n \mu^\dagger + \frac{1}{n}\sum_{i=1}^n \epsilon_i G_{n,\lambda_n}^{-1}\eta_i
\end{equation}
where
\begin{equation} \label{eq:Cons:Bound:Un}
U_n=\frac{1}{n}\sum_{i=1}^n \eta_i L_i.
\end{equation}
We bound $\|G_{n,\lambda_n}^{-1} U_n \mu^\dagger\|$ in Lemma~\ref{lem:Cons:Bound:NoNoiseBound} and $\|\frac{1}{n}\sum_{i=1}^n \epsilon_i G_{n,\lambda_n}^{-1}\eta_i\|$ in Lemma~\ref{lem:Cons:Bound:NoiseDifBound}.

In the proof of Lemma~\ref{lem:Cons:Bound:NoNoiseBound} we show that
$G_{n,\lambda_n}^{-1}:\text{Ran}(U_n)\to \text{Ran}(U_n)$.
Lemma~\ref{lem:Cons:Bound:USelfAdjointCompact} gives the conditions necessary to infer the existence of a orthonormal basis of eigenfunctions $\{\psi_j^{(n)}\}_{j=1}^\infty$ of $\text{Ran}(U_n)$.
Hence we can write
\[ \|G_{n,\lambda_n}^{-1}U_n\mu\|^2 = \sum_{j=1}^\infty (G_{n,\lambda_n}^{-1}U_n\mu,\psi_j^{(n)})^2. \]
From here we exploit the fact that $\psi_j^{(n)}$ are eigenfunctions.
We leave the details until the proof of Lemma~\ref{lem:Cons:Bound:NoNoiseBound}.

Lemma~\ref{lem:Cons:Bound:NoiseDifBound} is a consequence of being able to bound $\|G_{n,\lambda_n}^{-1}\|_{\mathcal{L}(\mathcal{H},\mathcal{H})}$ in terms of $\lambda_n$.
One is then left to show $\left(\frac{1}{n}\sum_{i=1}^n \epsilon_i\right)^2 =O(\frac{1}{n})$.
We start by showing that $U_n$ is compact, bounded, self-adjoint and positive semi-definite.

\begin{lemma}
\label{lem:Cons:Bound:USelfAdjointCompact}
Define $U_n$ by~\eqref{eq:Cons:Bound:Un}. Under
Assumptions~\ref{ass:Prelim:Spline:Space} and
\ref{ass:PreLim:Splines:etabound}, $U_n$ is almost surely a bounded,
self-adjoint, positive semi-definite and compact operator on $\mathcal{H}$.
\end{lemma}
\begin{proof}
In this proof we consider $\omega \in \Omega^\prime$ where $\Omega^\prime = \{ \omega: \|\eta_i(\omega)\|\leq
\alpha$ for all $i$\}, noting that $\mathbb{P}(\Omega^\prime)=1$ by Assumption \ref{ass:PreLim:Splines:etabound}.

Boundedness of $U_n$ follows easily as
\[ \|U_n\mu\| \leq \frac{1}{n} \sum_{i=1}^n \alpha^2 \|\mu\| = \alpha^2 \|\mu\|. \] 

Let $(\cdot,\cdot)_{\mathbb{R}^n}$ be the inner product on $\mathbb{R}^n$ given by
\[ ( x,y)_{\mathbb{R}^n} = \frac{1}{n} \sum_{i=1}^n x_i y_i \quad \quad \forall x,y\in\mathbb{R}^n. \]
Now for $x\in\mathbb{R}$ and $\nu\in \mathcal{H}$ we have
\[ ( x,L_i\nu )_{\mathbb{R}^1} = x L_i\nu = x (\eta_i,\nu) = (x\eta_i,\nu) \]
which shows $L^*_i:\mathbb{R}\to \mathcal{H}$ is given by $L^*_ix = x\eta_i$.
Now if we define $T_n=(L_1,\dots,L_n):\mathcal{H}\to\mathbb{R}^n$ then for $x\in\mathbb{R}^n$, $\nu\in\mathcal{H}$
\[ ( T_n\nu,x)_{\mathbb{R}^n} = \frac{1}{n} \sum_{i=1}^n L_i\nu x_i = \left(\frac{1}{n}\sum_{i=1}^n x_i \eta_i,\nu\right). \]
Hence $T_n^* x = \frac{1}{n} \sum_{i=1}^n x_i\eta_i$.
We have shown $U_n = T_n^* T_n$, and is therefore self-adjoint.

To show $U_n$ is positive semi-definite then we need
\[ (U_n\nu,\nu) \geq 0 \]
for all $\nu \in \mathcal{H}$.
This follows easily as
\[ (U_n\nu,\nu) = \frac{1}{n} \sum_{i=1}^n (L_i \nu)^2 \geq 0. \]

For compactness of $U_n$ (for $n$ fixed) let $\nu^m$ be a sequence with $\|\nu^m\|\leq 1$.
Since $|L_i \nu^m| \leq \alpha$ for every $\omega \in \Omega^\prime$, there exists a convergent subsequence $m_p$ such that
\[ L_i \nu^{m_p} \to \kappa_i \quad \forall i=1,2,\dots,n \quad \text{say.} \]
So $U_n \nu^{m_p} \to \frac{1}{n} \sum_{i=1}^n \eta_i \kappa_i \in\mathcal{H}$ as $m_p\to\infty$.
Therefore each $U_n$ is compact.
\end{proof}

Using the basis whose existence is implied by the previous lemma, we can bound the first term on the RHS of~\eqref{eq:Cons:Bound:decompmu}.

\begin{lemma}
\label{lem:Cons:Bound:NoNoiseBound}
Under Assumptions \ref{ass:Prelim:Spline:Space}-\ref{ass:PreLim:Splines:etabound} define $G_{n,\lambda_n}$ and $U_n$ by~\eqref{eq:Prelim:Splines:Gnlambda} and~\eqref{eq:Cons:Bound:Un} respectively.
Then with probability one we have
\[ \|G_{n,\lambda_n}^{-1} U_n\|_{\mathcal{L}(\mathcal{H},\mathcal{H})} \leq 1 \]
for all $n$.
\end{lemma}

\begin{proof}
First note that $\text{dim}(\text{Ran}(U_n))=\text{dim}(\text{span}\{\eta_1,\dots,\eta_n\})\leq n$.
Without loss of generality we will assume $\text{dim}(\text{Ran}(U_n))= n$ (else we can assume the dimension is $m_n$ where $m_n\leq n$ is an increasing sequence).
Clearly $\chi_1$ is a self-adjoint, bounded and compact operator on $\text{Ran}(U_n)$ as is $U_n$ by Lemma~\ref{lem:Cons:Bound:USelfAdjointCompact}.
Therefore there exists a simultaneous diagonalisation of $U_n$ and $\chi_1$ on $\text{Ran}(U_n)$, i.e. there exists $\beta_j^{(n)},\gamma_j^{(n)}$ and $\psi_j^{(n)}$ such that
\[ U_n \psi^{(n)}_j = \beta_j^{(n)} \psi^{(n)}_j \quad \text{and} \quad \chi_1 \psi^{(n)}_j = \gamma_j^{(n)} \psi^{(n)}_j \]
for all $j=1,2,\dots,n$.
Since $\chi_1$ is the projection operator then we must have $\gamma_j^{(n)}\in \{0,1\}$. 
Furthermore $\psi_j^{(n)}$ form an orthonormal basis of $\text{Ran}(U_n)$.
Since $U_n$ is positive semi-definite it follows that $\beta_j^{(n)}\geq 0$.
We have
\[ G_{n,\lambda_n} \psi_j^{(n)} = U_n\psi^{(n)}_j + \lambda_n \chi_1 \psi_j^{(n)} = \left( \beta_j^{(n)} + \lambda_n \gamma_j^{(n)} \right) \psi_j^{(n)}. \]
So,
\[ G_{n,\lambda_n}^{-1} \psi_j^{(n)} = \frac{1}{\beta_j^{(n)} + \lambda_n} \psi_j^{(n)}. \]
In particular this shows that
\[ G_{n,\lambda_n}^{-1} U_n: \mathcal{H} \to \text{Ran}(U_n). \]
Assume $\mu\in \mathcal{H},\nu\in \text{Ran}(U_n)$, then
\[ \mu = \sum_{i=1}^n (\mu,\psi_i^{(n)}) \psi_i^{(n)} + \hat{\mu} \quad \quad \text{and} \quad \quad \nu = \sum_{i=1}^n (\nu,\psi_i^{(n)}) \psi_i^{(n)} \]
where $\hat{\mu}\in \text{Ran}(U_n)^\bot$.
Therefore,
\begin{align*}
(U_n\mu,\psi_j^{(n)}) & = \sum_{i=1}^n (\mu,\psi_i^{(n)}) (U_n\psi_i^{(n)},\psi^{(n)}_j) = \beta^{(n)}_j (\mu,\psi_j^{(n)}) \\
(G_{n,\lambda_n}^{-1} \nu,\psi^{(n)}_j) & = \sum_{i=1}^n (\nu,\psi_i^{(n)}) (G_{n,\lambda_n}^{-1} \psi_i^{(n)},\psi_j^{(n)}) = \frac{1}{\beta_j^{(n)}+\lambda_n\gamma_j^{(n)}} (\nu,\psi_j^{(n)}).
\end{align*}
Which implies
\[ (G_{n,\lambda_n}^{-1} U_n \mu,\psi^{(n)}_j) = \frac{1}{\beta_j^{(n)}+ \lambda_n\gamma_j^{(n)}} (U_n\mu,\psi_j^{(n)}) = \frac{\beta_j^{(n)}}{\beta_j^{(n)} + \lambda_n \gamma_j^{(n)}} (\mu,\psi_j^{(n)}). \]
Hence
\begin{align*}
\| G_{n,\lambda_n}^{-1} U_n\mu \|^2 & = \sum_{j=1}^n (G_{n,\lambda_n}^{-1} U_n \mu,\psi^{(n)}_j)^2 \\
 & = \sum_{j=1}^n \left(\frac{\beta_j^{(n)}}{\beta_j^{(n)}+\lambda_n \gamma_j^{(n)}}\right)^2 (\mu,\psi_j^{(n)})^2 \\
 & \leq \sum_{j=1}^n (\mu,\psi_j^{(n)})^2 \\
 & \leq \| \mu\|^2.
\end{align*}
This proves the lemma.
\end{proof}

We now focus on bounding $\| G_{n,\lambda_n}^{-1} \nu_n\|$ where $\nu_n=\frac{1}{n} \sum_{i=1}^n \epsilon_i \eta_i$.

\begin{lemma}
\label{lem:Cons:Bound:NoiseDifBound}
Under Assumptions~\ref{ass:Prelim:Spline:Space}-\ref{ass:Cons:lambdascale} define $G_{n,\lambda_n}$ by~\eqref{eq:Prelim:Splines:Gnlambda}.
Then
\[ \mathbb{E}\left[\left.\left\|\frac{1}{n} \sum_{i=1}^n \epsilon_i G_{n,\lambda_n}^{-1} \eta_i \right\|^2 \right|\mathcal{G}_n\right] = O(1) \quad \text{almost surely.} \]
\end{lemma}

\begin{proof}
Recalling $B$ from the proof of Lemma~\ref{lem:Prelim:Splines:Minfn}, we have
\[ (G_{n,\lambda_n}\mu,\mu) = B(\mu,\mu) \geq \lambda_n \|\mu\|_1^2. \]
This implies $\|G_{n,\lambda_n} \mu\|\geq \lambda_n \|\mu\|_1$.
By Lemma~\ref{lem:Prelim:Splines:Minfn} there exists a well defined inverse of
$G_{n,\lambda_n}$ at $\eta_i$, hence we let $\mu = G_{n,\lambda_n}^{-1} \eta_i$ and we have 
\[ \| G_{n,\lambda_n}^{-1} \eta_i \|_1 \leq \frac{1}{\lambda_n} \|\eta_i\| \leq \frac{\alpha}{\lambda_n}. \]
almost surely. Now, define $\nu_n = \frac{1}{n}\sum_{i=1}^n \epsilon_i \eta_i$ and 
\begin{align*}
\mathbb{E}\left[\left. \left\| G_{n,\lambda_n}^{-1} \nu_n \right\|_1^2 \right|\mathcal{G}_n\right] & \as \frac{\sigma^2}{n^2} \sum_{i=1}^n \left\| G_{n,\lambda_n}^{-1} \eta_i \right\|_1^2 \\
 & \leq \frac{\alpha^2\sigma^2}{n\lambda_n^2}.
\end{align*}
Combined with Lemma~\ref{lem:Cons:Bound:H0Bound} (the $\mathcal{H}_0$ bound)
this proves the lemma.
\end{proof}

%

Recalling~\eqref{eq:Cons:Bound:decompmu} and via
Lemmas~\ref{lem:Cons:Bound:NoNoiseBound}
and~\ref{lem:Cons:Bound:NoiseDifBound} we 
obtain the following asymptotic bound on minimizers in $\mathcal{H}$.

\begin{theorem}
\label{thm:Cons:Bound:muBound}
Under Assumptions~\ref{ass:Prelim:Spline:Space}-\ref{ass:Cons:lambdascale} we have
\begin{equation} \label{eq:Cons:Bound:munbound}
\mathbb{E} \left[ \|\mu^n \|^2 |\mathcal{G}_n \right] = O(1) \quad \text{almost surely.}
\end{equation}
\end{theorem}

This is a stronger result than we needed; we were only required to show that $\|\mu^n\|$ is bounded in probability.
Taking expectation of~\eqref{eq:Cons:Bound:munbound} one has
\[ \mathbb{E} \|\mu^n\|^2 = O(1). \]
Hence applying Chebyshev's inequality we may conclude that $\|\mu^n\|=O_p(1)$.

\begin{corollary}
\label{cor:Cons:Bound:Boundinprob}
Under Assumptions~\ref{ass:Prelim:Spline:Space}-\ref{ass:Cons:lambdascale} we have $\|\mu^n\|=O_p(1)$.
\end{corollary}


We conclude this section with a brief analysis of the rate of convergence.
For any $F\in\mathcal{H}^*$, by the Riesz Representation Theorem, there exists $\xi\in\mathcal{H}$ such that $F(\mu)=(\mu,\xi)$ for all $\mu\in\mathcal{H}$.
Hence
\[ F(\mu^n) - F(\mu^\dagger) = ((G_{n,\lambda_n}^{-1} U_n - \text{Id}) \mu^\dagger + G_{n,\lambda_n}^{-1} \nu^n, \xi) \]
where $\nu^n= \frac{1}{n} \sum_{i=1}^n \epsilon_i \eta_i$.
Decomposing $\mathcal{H}$ into $\mathcal{H}=\overline{\text{Ran}(U_n)}\oplus \text{Ran}(U_n)^\bot $ one can write
\begin{align}
F(\mu^n) - F(\mu^\dagger) & = \left(\left(G_{n,\lambda_n}^{-1} U_n - \chi_{\overline{\text{Ran}(U_n)}}\right) \mu^\dagger,\xi\right) - \left(\chi_{\text{Ran}(U_n)^\bot}\mu^\dagger,\xi\right) + \left( G_{n,\lambda_n}^{-1} \nu^n, \xi\right) \notag \\
 & = \sum_{j=1}^n \frac{-\lambda_n}{\beta_j^{(n)}+ \lambda_n}\left(\mu^\dagger,\psi_j^{(n)}\right)\left(\psi_j^{(n)},\xi\right) - \left(\chi_{\text{Ran}(U_n)^\bot}\mu^\dagger,\xi\right) + \left( G_{n,\lambda_n}^{-1} \nu^n, \xi\right) \label{eq:Cons:Bound:RateExpansion}
\end{align}
where $\chi_{\overline{\text{Ran}(U_n)}}$ is the projection onto $\overline{\text{Ran}(U_n)}$.
If we assume 
\begin{equation} \label{eq:Cons:Bound:HScale}
\lim_{n\to \infty} \sum_{j=1}^n \frac{1}{\beta_j^{(n)}} <\infty. 
\end{equation}
Then
\[ \sum_{j=1}^n \frac{-\lambda_n}{\beta_j^{(n)}+ \lambda_n}\left(\mu^\dagger,\psi_j^{(n)}\right)\left(\psi_j^{(n)},\xi\right) \leq \|\mu^\dagger\| \|\xi\| \lambda_n \sum_{j=1}^n \frac{1}{\beta_j^{(n)}}.  \]
And therefore the first term in~\eqref{eq:Cons:Bound:RateExpansion} is of the order $n^{-p}$.
By the proof of Lemma~\ref{lem:Cons:Bound:NoiseDifBound} the third term in~\eqref{eq:Cons:Bound:RateExpansion} is of order $\frac{1}{\sqrt{n}\lambda_n}$.
The second term is independent of $\lambda_n$.
The optimal rate of convergence is therefore found by balancing the first and third terms.
This will imply an optimal choice of $p=\frac{1}{4}$.
We summarise in the following proposition.


\begin{proposition}
\label{prop:Cons:Bound:OptimalRate}
Under Assumptions~\ref{ass:Prelim:Spline:Space}-\ref{ass:Cons:Norm}, for $F\in\mathcal{H}^*$ take $\xi\in\mathcal{H}$ such that $F(\mu)=(\mu,\xi)$ and assume~\eqref{eq:Cons:Bound:HScale} holds and that there exists $q>0$ such that 
\[ \left| \|\mu^\dagger\|-\|\chi_{\mathrm{Ran}(U_n)} \mu^\dagger\| \right| \lesssim n^{-q} \]
where $U_n$ is defined by~\eqref{eq:Cons:Bound:Un} and $(\beta_j^{(n)},\psi_j^{(n)})$ are an eigenvalue-eigenfunction pair for $U_n$.
Then 
\begin{equation} \label{eq:Cons:Bound:Rate}
\mathbb{E}\left[ | F(\mu^n) - F(\mu^\dagger)| \; |\mathcal{G}_n \right] = O\left(n^{-p}\right) + O\left(n^{-q}\right) + O\left(\frac{1}{\lambda_n\sqrt{n}}\right) \quad \quad \text{almost surely.}
\end{equation}
In particular the optimal choice is $p = \frac{1}{4}$ in which case the rate of convergence is
\[ \mathbb{E}\left[ | F(\mu^n) - F(\mu^\dagger)| \; |\mathcal{G}_n \right] = O\left( n^{ \max\{-\frac{1}{4},-q\}} \right). \]
\end{proposition}

\begin{proof}
The argument preceding the theorem provides the proof for the first term in~\eqref{eq:Cons:Bound:Rate} and the third term is a consequence of Lemma~\ref{lem:Cons:Bound:NoiseDifBound}.
The second term follows easily from
\[ \left| \left(\chi_{\text{Ran}(U_n)^\bot}\mu^\dagger,\xi\right) \right| \leq \|\xi\| \|\chi_{\text{Ran}(U_n)^\bot}\mu^\dagger \| \leq \|\xi\| \left( \|\mu^\dagger\| - \|\chi_{\text{Ran}(U_n)}\mu^\dagger \| \right). \]

%

The optimal rate is a consequence of choosing $p$ that minimizes $n^{-p} + n^{p-0.5}$. 
\end{proof}

The conditions of the above theorem are difficult to theoretically verify.
Even for the special spline problem the authors know of no method to check whether assumption~\eqref{eq:Cons:Bound:HScale} holds and whether such a $q$ exist.
We leave further investigation into the rate of convergence for future works.

%
\subsection{Sharpness of the Scaling Regime - Proof of Theorem~\ref{thm:Cons:Sharp} \label{subsec:Cons:Sharp}}
\begin{proof}[Proof of Theorem~\ref{thm:Cons:Sharp}]
Fix any $\alpha>0$ and without loss of generality we can choose $\{\eta_t\}_{t\in\mathcal{I}}$ such that $\|\eta_t\|=\alpha$ for all $t\in\mathcal{I}$.
Define $L_t\in\mathcal{H}$ by $L_t = (\eta_i,\cdot)$.

In the proof of Lemma~\ref{lem:Prelim:Splines:Minfn} we showed
\[ \left| (G_{n,\lambda_n}\mu,\nu) \right| \leq (\alpha^2 + \lambda_n) \| \mu\| \|\nu\|. \]
Letting $\nu=G_{n,\lambda_n}\mu$, for $\mu\in\text{span}\{\eta_1,\dots,\eta_n\}$, one has
\[ \|G_{n,\lambda_n} \mu\|^2 \leq (\alpha^2 + \lambda_n) \|\mu\| \|G_{n,\lambda_n} \mu\|. \]
And hence
\[ \|G_{n,\lambda_n} \mu\| \leq (\alpha^2 + \lambda_n) \|\mu\|. \]
Which implies
\[ \| G^{-1}_{n,\lambda_n} \mu\| \geq \frac{1}{\alpha^2 + \lambda_n} \|\mu\|. \]
Now, for $\nu^n=\frac{1}{n}\sum_{i=1}^n \epsilon_i \eta_i$, we consider
\begin{align*}
\mathbb{E}\left[\left. \|G_{n,\lambda_n}^{-1} \nu^n\|^2 \right|\mathcal{G}_n\right] & \geq \frac{1}{(\alpha^2+\lambda_n)^2} \mathbb{E} \left[ \|\nu^n\|^2 |\mathcal{G}_n \right] \\
 & \as \frac{\sigma^2\alpha^2}{\lambda_n^2 n (\alpha^2+\lambda_n)^2} \\
 & \to \infty
\end{align*}
as $\lambda_n^2 n\to 0$.
Hence by taking expectations:
\[ \mathbb{E}\left[\|G_{n,\lambda_n}^{-1} \nu^n\|^2 \right] \to \infty. \]
By noting
\[ \mathbb{E}\left[ \|\mu^n\|^2 \right] = \mathbb{E}\left[ \|G_{n,\lambda_n}^{-1} U_n \mu^\dagger \|^2\right] + \mathbb{E}\left[ \|G_{n,\lambda_n}^{-1} \nu^n\|^2\right] \]
we conclude the proof.
\end{proof}
\section{Application to the Special Spline Model \label{sec:Ex}}
Consider the application to the special spline case, $L_i\mu=\mu(t_i)$.
We let
\[ \mathcal{H} = H^m := \left\{ g:[0,1] \to \mathbb{R} \text{ s.t } \nabla^i g \text{ abs. cts. for } i=1,2,\dots,m-1 \text{ and } \nabla^m g\in L^2 \right\}. \]
For $m\geq 1$, $\mathcal{H}$ is a reproducing kernel Hilbert space and therefore $L_i$ as defined are linear and bounded operators on $\mathcal{H}$.
See~\cite{bogachev98,wahba90} for more details on reproducing kernel Hilbert spaces.
The special spline solution is the minimizer of
\[ f_n(\mu) = \frac{1}{n} \sum_{i=1}^n (y_i - \mu(t_i))^2 + \lambda_n \| \nabla^m \mu \|_{L^2}^2 \]
over all $\mu \in H^m$.
It can be shown that the minimizer $\mu^{(n)}$ of $f_n$ is a piecewise polynomial of degree $2m-1$ in each interval $(t_i,t_{i+1})$ for $i=0,\dots,n$ (where we define $t_0=0$ and $t_{n+1}=1$), for example see~\cite[Section 1.3]{wahba90}.

This section discusses the following points.
\begin{enumerate}
\item The decomposition $\mathcal{H}=\mathcal{H}_0\oplus\mathcal{H}_1$ where $\mathcal{H}_0$ is finite dimensional.
\item The function $\eta_t$ corresponding to $(\eta_t,\mu)=L_t\mu=\mu(t)$.
\end{enumerate}
The other assumptions needed to apply Theorem~\ref{thm:Cons:Cons} are Assumption~\ref{ass:Prelim:Spline:Lunique} and Assumption~\ref{ass:Cons:Norm}.
Assumption~\ref{ass:Prelim:Spline:Lunique} is
\[ \mu(t)=\mu(r) \quad \text{for all polynomials } \mu \text{ of degree at most } m-1 \text{ then } t=r \]
which clearly holds.
Assumption~\ref{ass:Cons:Norm} becomes
\[ \int_0^1 |\nu(t)|^2 \phi_T(\mathrm{d} t) = 0 \Leftrightarrow \nu=0 \]
which, for example, is true if $\phi_T(\mathrm{d} t) = \hat{\phi}_T(t) \; \mathrm{d} t$ and $\hat{\phi}_T(t)>0$ for all $t\in [0,1]$.


\paragraph*{1. The decomposition $\mathcal{H}=\mathcal{H}_0\oplus\mathcal{H}_1$.}
For $\mu\in\mathcal{H}$ by Taylor expanding $\mu$ from 0 we can write:
\begin{align*}
\mu(t) & = \sum_{i=0}^{m-1} \frac{\nabla^i\mu(0)}{i!} t^i + R(t)
\end{align*}
where $\nabla^i R(0)=0$ for all $i=0,1,\dots,m-1$.
Hence $R\in \mathcal{H}_1$ where
\[ \mathcal{H}_1 = \left\{ g\in H^m: \nabla^i g(0) = 0 \text{ for all } i=0,1,\dots,m-1 \right\}. \]
A Poincar\'e inequality holds on this space so $\|\mu\|_1^2=\int_0^1 |\nabla^m \mu(t)|^2 \; \mathrm{d} t$ is a norm on $\mathcal{H}_1$.

We define $\mathcal{H}_0$ to be the span of the functions $\zeta_i$ defined by
\[ \zeta_i(t) = \frac{t^i}{i!} \quad \quad \quad \text{for } i=0,1,\dots,m-1. \]
The space is equipped with the inner product
\[ ( \mu,\nu )_0 = \sum_{i=0}^{m-1} \nabla^i \mu(0) \nabla^i \nu(0). \]
The space $\mathcal{H}_0$ has $\text{dim}(\mathcal{H}_0)=m$.

\paragraph*{2. The functions $\eta_t$.}
In the above $R$ is given by
\begin{align*}
R(t) & = \int_0^1 \frac{(t-u)_+^{m-1}}{(m-1)!} \nabla^m\mu(u) \; \mathrm{d} u = \int_0^1 G(t,u) \nabla^m \mu (u) \; \mathrm{d} u
\end{align*}
where $(u)_+=\max\{0,u\}$ and
\[ G(t,u) = \frac{(t-u)_+^{m-1}}{(m-1)!} \]
is the Green's function for $\nabla^m \mu =\nu$ and boundary conditions $\nabla^j \mu(0)=0$ for all $0\leq j\leq m-1$.

We claim that $\eta_t\in H^m$ satisfying $(\eta_t,\mu)=\mu(t)$ are given by
\[ \eta_t(r) = \sum_{i=0}^{m-1} \zeta_i(t)\zeta_i(r) +  \int_0^1 G(t,u) G(r,u) \; \mathrm{d} u =: \eta_t^0(r) + \eta_t^1(r).  \]
Furthermore $\eta^0_t\in \mathcal{H}_0$ and $\eta_t^1\in \mathcal{H}_1$ for all $t\in [0,1]$.
The proof follows directly from calculating
\[ (\eta_t,\mu) = \sum_{i=0}^{m-1} \nabla^i\eta_t(0) \nabla^i\mu(0) + \int_0^1 \nabla^m \eta_t(u) \nabla^m\mu (u) \; \mathrm{d} u \]
and noticing
\begin{align*}
\nabla^i \eta_t(r) & = \sum_{j=1}^{m-1} \zeta_j(t) \left[\nabla^i\zeta_j(r)\right]_{r=0} = \zeta_i(t) \quad \text{for } i<m \\
\nabla^m \eta_t(r) & = \nabla^m_r \int_0^1 G(t,u) G(r,u) \; \mathrm{d}u = G(t,r).
\end{align*}
One can easily show that $\|\eta_t\|\leq 1$ for all $t\in [0,1]$.

Continuity of $\eta_t$ follows easily.
As each polynomial is Lipschitz continuous on the interval $[0,1]$, there exists a constant $C_i$ (depending on the order of the polynomial $i$) such that $|\zeta_i(t) - \zeta_i(s)|\leq C_i|t-s|$.
Now for the integral term let $m\geq 2$ and $s\geq t$ then:
\begin{align*}
\left|\int_0^1 \left( G(s,u) - G(t,u) \right) G(r,u) \; \mathrm{d} u \right| & = \left| \int_0^1 \left( \mathbb{I}_{s>u} \frac{(s-u)^{m-1}}{(m-1)!} - \mathbb{I}_{t>u} \frac{(t-u)^{m-1}}{(m-1)!} \right) G(r,u) \; \mathrm{d}u \right| \\
 & \leq \int_t^s \frac{(s-u)^{m-1}}{(m-1)!}G(r,u) \; \mathrm{d} u \\
 & \quad \quad \quad + \frac{1}{(m-2)!} \int_0^t \left| s-t\right| g(r,u) \; \mathrm{d} u \\
 & \leq \frac{m|s-t|}{[(m-1)!]^2}.  
\end{align*}
The case $m=1$ is similar.
It follows that $\|L_s-L_t\|_{\mathcal{H}^*} = \|\eta_s - \eta_t\| \leq C|s-t|$ for some $C<\infty$ and hence $L_t$ is continuous.

\section*{Acknowledgments}

This work was carried out whilst MT was part of MASDOC at the University of Warwick and supported by an EPSRC Industrial CASE Award PhD Studentship with Selex ES Ltd.


\bibliographystyle{plain}
\bibliography{references}

\end{document}